\documentclass[12pt]{amsart}
\usepackage{times}
\usepackage{amsmath,  amssymb,slashed,url,bm,upgreek,amsthm}
\usepackage{graphicx,enumerate}
\usepackage{hyperref}
\newcounter{intro}

\newtheorem{theo}[intro]{Theorem}
\newtheorem{coro}[intro]{Corollary}

\newtheorem{thm}{Theorem}[section]
\newtheorem{lem}[thm]{Lemma}

\newtheorem{cor}[thm]{Corollary}
\newtheorem{defi}[thm]{Definition}
\newtheorem{rem}[thm]{Remark}
\newtheorem{rems}[thm]{Remarks}

\newtheorem*{merci}{Acknowledgements}

\newcommand{\cref}[1]{Corollary~\ref{#1}}

\newcommand{\lref}[1]{Lemma~\ref{#1}}

\newcommand{\tref}[1]{Theorem~\ref{#1}}

\DeclareMathOperator{\un}{\mathbf{1}}
\DeclareMathOperator{\eucl}{eucl}

\DeclareMathOperator{\ricci}{Ricci}
\DeclareMathOperator{\ricm}{Ric_-}
\DeclareMathOperator{\scal}{Scal}

\DeclareMathOperator{\dv}{dv}

\DeclareMathOperator{\vol}{vol}
\def\R{\mathbb R}\def\N{\mathbb N}\def\bS{\mathbb S}\def\bB{\mathbb B}\def\bC{\mathbb C}

\def\cC{\mathcal C}
\def\cD{\mathcal D}

\def\sch{Schr\"odinger }

\def\cU{\mathcal U}

\begin{document}
\title[Euclidean volume growth]  {Euclidean volume growth for complete Riemannian manifolds }
\author{ Gilles Carron }
\address{Laboratoire de Math\'ematiques Jean Leray (UMR 6629), Universit\'e de Nantes, CNRS,
2, rue de la Houssini\`ere, B.P.~92208, 44322 Nantes Cedex~3, France}
\email{Gilles.Carron@univ-nantes.fr}
\begin{abstract}
We provide an overview of technics that lead to an Euclidean upper bound on the volume of geodesic balls.  \par
R\'ESUM\'E: 
Nous donnons un aper\c cu des techniques qui conduisent \`a une borne sup\'erieure euclidienne sur le volume des boules g\'eod\'esiques.
\end{abstract}
\date\today
\maketitle

\footnotetext{\emph{Mots cl\'es: }  croissance du volume,  }

\section{Introduction}
In this  paper, we survey a number of recent results concerning the following question:  when does a complete Riemannian manifold $(M^n,g)$ has Euclidean volume growth, i.e. we are looking for estimates of the type 
\begin{equation}\label{EVG}
\tag{EVG}  \forall R>0 \ \colon\  \vol B(x,R)\le C R^n
\end{equation}
where the constant $C$ may depend on the point $x$ or not.  
 We will also obtain some new results and will give several examples that illustrate the optimality of certains of these results.
 
Such an estimate has some important consequences:
\begin{enumerate}[i)]
\item A complete Riemannian surface $(M^2,g)$ satisfying (\ref{EVG})  is parabolic. That is to say $(M^2,g)$ has no positive Green kernel: there is no  $G\colon M\times M\setminus\mathrm{Diag}\longrightarrow (0,\infty)$ such that $\Delta_y G(x,y)=\delta_x(y)$. We recommend the beautiful and very comprehensive survey on parabolicity written by A. Grigor'yan \cite{GriBAMS}. In dimension $2$, parabolicity is a conformal property and a parabolic surface with finite topological type\footnote{that is homeomorphic to the interior of a compact surface with boundary.} is conformal to a closed surface with a finite number of points removed : there is a closed Riemannian  surface $(\overline{M},\bar g)$, a finite set $\{p_1,\dots, p_\ell\}\subset \overline{M}$ and a smooth function $f\colon \overline{M}\setminus\{p_1,\dots, p_\ell\}\longrightarrow\R$ such that $(M^2,g)$ is isometric to $\left(\overline{M}\setminus\{p_1,\dots, p_\ell\}, e^{2f}\bar g\right).$ 
\item In higher dimension, the condition  (\ref{EVG}) implies that the manifold is $n-$parabolic. It is a non linear analogue of the parabolicity (\cite{CHS,holo1,holo2}).
\item According to R. Schoen, L. Simon and S-T. Yau \cite{SSY}, if a complete stable minimal hypersurface $\Sigma\subset \R^{n+1}$ with $n\in \{2,3,4,5\}$ satisfies the Euclidean volume growth (\ref{EVG}) , then $\Sigma$ is an affine hypersurface. In dimension $n=2$, M. Do Carmo and C.K. Peng proved  that a stable minimal surface in $\R^3$ is planar \cite{dCPeng}. But nothing is known in higher dimension.
\item If $M^n$ is the universal cover of a closed Riemannian manifold  $ \breve{M}$ and satisfies the Euclidean volume growth (\ref{EVG}) , then the fundamental group of $\breve{M}$ is virtually nilpotent \cite{Gromov}. 
\item Another topological implication is that if a complete Riemannian manifold $(M^n,g)$ is doubling: there is a uniform constant $\upgamma$ such that for any $x\in M$ and $R>0$: $\vol B(x,2R)\le C\upgamma \vol B(x,R)$, then $M$ has only a finite number of ends, that is to say there is a constant $N$ depending only of $\upgamma$ such that for any $K\subset M$ compact subset of $M$, $M\setminus K$ has at most $N$ unbounded connected components (\cite{CarronRMI}). In particular if $(M^n,g)$ satisfies a uniform upper and lower Euclidean volume growth: for any $x\in M$ and $R>0$: $$\uptheta^{-1} R^n\le  \vol B(x,R)\le \uptheta R^n$$ then $M$ has a finite number of ends.
\item In (\cite{TV1}), G. Tian and J. Viaclovsky  have obtained that if $(M^n,g)$ is a complete Riemannian manifold such that 
\begin{itemize}
\item $\forall x\in M, \forall R>0\colon \vol B(x,R)\ge cR^n$.
\item $\|\mathrm{Rm}\|(x)=o\left(d(o,x)^{-2}\right)$
\end{itemize} then $(M^n,g)$  satisfies  (\ref{EVG}) and it is an Asymptotically Locally Euclidean  space. This result was a key point toward the description of the moduli spaces of critical Riemannian metrics on manifolds of dimension $4$ (\cite{TV2}).

\end{enumerate}

We will review $3$ different technics that leads to (\ref{EVG}).
\begin{enumerate}[I-]
\item Comparison theorem and elaborations from the classical Bishop-Gromov comparison theorem.
\item Spectral theory and elaborations from a result of P. Castillon.
\item Harmonic analysis and the relevance of the concept of Strong $A_\infty$ weights of G. David and S. Semmes for conformal metrics.
\end{enumerate} 
In the next section,  we first give a short overview of these technics. More details and some proofs of new results will be given in specific sections. New results will labelled by letters (A,B...). The third section is devoted to news results obtained with comparison technics, the fourth section is devoted to the presentation of the results obtained from spectral theory, the application of  Strong $A_\infty$ weights  is described in the fifth section. The last section will be devoted to construction of news examples.

\begin{merci} I wish to thank S. Gallot and H. Rosenberg   for useful suggestions.   I thank the Centre Henri Lebesgue ANR-11-LABX-0020-01 for creating an attractive mathematical environment. I was partially supported by the ANR grants:  {\bf ANR-17-CE40-0034}: {\em CCEM} and {\bf ANR-18-CE40-0012}: {\em RAGE}.
\end{merci}
\section{Overview of the different technics and results}

\subsection{Comparison theorem}
 When $(M^n,g)$ is a complete Riemannian manifold, we defined $\ricm\colon M\longrightarrow \R_+$ by 
$\ricm(x)=0$ if $\ricci(x)\ge 0$ and if  $\ricci(x)$ has a negative eigenvalue then  $-\ricm(x)$ is the lowest eigenvalue of  $\ricci(x)$. Hence on a manifold with non negative Ricci curvature, we have $\ricm=0$. S. Gallot, P. Li and S-T. Yau, P. Petersen and G. Wei,  E. Aubry have obtained some refinement of the Bishop-Gromov comparison theorem under some integral bound on the negative part of the Ricci curvature \cite{Gallot,LY,PW,Aubry}. The proof of these volume estimates leads to the following new result:
\begin{theo}\label{logvol} Let $(M^n,g)$ be a complete Riemannian manifold of dimension $n\ge 3$. Assume that there is some  $\nu>n$ such that:
$$\int_M \ricm^{\frac{n}{2}}\dv<\infty\ \mathrm{and}\ \int_M \ricm^{\frac{\nu}{2}}\dv<\infty,$$ then there is a $R_0$ depending only on $n,\nu$, $\|\ricm\|_{L^\frac{n}{2}}$ and $\|\ricm\|_{L^\frac{\nu}{2}}$ such that if $x\in M$ then
$$\vol B(x,R)\le 2\omega_n R^n, R\le R_0$$
and 
$$\vol B(x,R)\le C(n,\nu) R^n\left(\log\left(\frac{2R}{R_0}\right)\right)^{\frac{n}{2}-1}, R\ge R_0.$$
\end{theo}
\begin{rem}
 The statement is new but the proof follows from the one of S. Gallot, P. Li, S-T. Yau, P. Petersen, G. Wei and E. Aubry. \end{rem}
This result has the following corollary 
\begin{coro} In the setting of \tref{logvol}, the Riemannian manifold $(M^n,g)$ is $n-$parabolic.
\end{coro}

And this volume estimate also gives an improvement of \cite[Theorem 2.1]{CH}:
\begin{coro}\label{CH} Let $\Omega$ be a domain of $(M,g_0)$ a compact Riemannian manifold of dimension $n>2$. Assume $\Omega$ is endowed with a complete Riemannian metric $g$ which is conformal to $g_0$. Suppose moreover that for some $\nu>n:$
$$\int_M \|\ricci\|^{\frac{n}{2}}\dv_g<\infty\ \mathrm{and}\ \int_M |\ricm|^{\frac{\nu}{2}}\dv_g<\infty$$
Then there is a finite set $\{p_1,\dots,p_k\}\subset M$  such that
$$\Omega=M\setminus \{p_1,\dots,p_k\},$$
Moreover $(\Omega,g)$ satisfies the Euclidean volume growth (\ref{EVG}). \end{coro}
\begin{rems}\begin{itemize}
\item The hypotheses of \cite[Theorem 2.1]{CH} required moreover the estimate 
$$\vol_g B(o,R)=\mathrm{o}\left( R^n \log^{n-1}(R) \right).$$
According to \tref{logvol}, this volume estimate is implied by the other hypotheses.
\item The Euclidean volume growth is a consequence of \cite[theorem 1.6 ]{ACT} (see also \tref{HARn}-b).
\end{itemize} \end{rems}
We will give examples that illustrate that the conclusions of \tref{logvol} and \cref{CH} are optimal.
When $g$ is a Riemannian metric on a manifold $M$, the function $\sigma_-(g)\colon M\rightarrow \R_+$ is defined by 
$ \sigma_-(g)(x)=0$ if all the sectional curvature at $x$ are non negative and in the other case, $-\sigma_-(g)(x)$ is the lowest of the sectional curvature of $g$ at $x$.
\begin{theo}\label{Exvol}For any $n\ge 3$ and $R>3$,  there is a complete conformal metric $g_R=e^{2f_R}\mathrm{eucl}$ on $\R^n$ whose sectional curvatures satisfy: 
$$\sigma_-(g_R)\le C(n) \  \mathrm{and}\ \int_{\R^n} \sigma_-(g_R)^{\frac n2} \dv_{g_R}\le C(n)$$ and such that
$$\vol_{g_R}\left(B(o,R)\right)\ge R^n\left(\log R \right)^{\frac{n}{2}-1}/C(n),$$
where the positive constant $C(n)$ depends only on $n$
\end{theo}
\begin{theo}\label{ExCH}If $n\ge 3$, there is an infinite set $\Sigma\subset \bS^n$ and a complete conformal metric $g=e^{2f}\mathrm{can}$ on $\bS^n\setminus \Sigma$ 
whose sectional curvature are bounded from below  and such that 
$$\int_{\bS^n\setminus \Sigma} \sigma_-(g)^{\frac n2} \dv_{g}<\infty.$$\end{theo}
These constructions are slight modifications of examples furnished by S. Gallot and E. Aubry (\cite{Gallot,Aubry}).
\subsection{Harmonic analysis}
The Euclidean volume growth (\ref{EVG}) result in \cref{CH} is in fact a consequence of the following result (\cite{ACT})
\begin{thm}\label{harmonicanalysis} Let $g = e^{2f} \eucl$ be a conformal deformation of the Euclidean metric on $\R^n$ such that:
\begin{itemize}
\item $\vol(\R^n, g) = +\infty$,
\item $\displaystyle \int_{\R^n} | \scal_g |^{n/2}\dv_g < +\infty$.
\end{itemize}
Then there is constant $C$ such that any $g$-geodesic ball $B_g(x,R)\subset \R^n$ satisfies
$$C^{-1} R^n\le \vol_g B_g(x,R)\le C R^n.$$
 Hence $(\R^n,g)$ satisfies the Euclidean volume growth (\ref{EVG}).
\end{thm}
The constant C here does not only depend on $\|\scal_g\|_{L^{\frac n 2}} $; but there is some $\upepsilon_n>0$ and some $C(n)$  such that if $\|\scal_g\|_{L^{\frac n 2}}<\upepsilon_n$ then any $g$-geodesic ball $B_g(x,R)\subset \R^n$ satisfies
$$C(n)^{-1} R^n\le \vol_g B_g(x,R)\le C(n) R^n.$$

The \tref{Exvol} shows the importance of the hypothesis on the control of the positive part of the scalar curvature.

This result is obtained using real harmonic analysis tools and in particular the notion of the strong $A_\infty$ weights which were introduced by G. David and S. Semmes (\cite{DS}). The original motivation was to find a characterization of weights that are  comparables with a quasiconformal Jacobian. The result of \cite{ACT} has been inspired by a similar study of Y. Wang who obtained in \cite{Wang} a similar result based on the $L^1$ norm of $\mathbf{Q}$ of the metric $g$, that is of $$\int_{\R^n} \left| \Delta^{\frac n2} f\right|(x) dx.$$

\subsection{Spectral theory} 
The study of volume growth estimate through spectral theory is motivated by the above question iii) about stable minimal hypersurfaces. Indeed let $M^n$ be a complete stable minimal hypersurface immersed in the Euclidean space $\R^{n+1}$ and let $\mathbf{I\!I}$ be its second fundamental form, the stability condition says that the \sch operator $\Delta_g-|\mathbf{I\!I}|^2$ is a non-negative operator, that is to say 
$$\int_M |\mathbf{I\!I}|^2\upvarphi^2\dv_g\le \int_M |d\upvarphi|_g^2\dv_g,\ \forall \upvarphi\in \cC_0^\infty(M).$$
But the Gauss-Egregium theorem implies that 
$$\ricci(\xi,\xi)=-\langle \mathbf{I\!I}(\xi), \mathbf{I\!I}(\xi)\rangle.$$
In particular, we have $$\ricm(x)\le \frac{n-1}{n} |\mathbf{I\!I}|^2$$ and the stability condition implies that \sch operator $\Delta-\frac{n}{n-1}\ricm$ is non negative.

In dimension $2$, a very satisfactory answer is given by the following very beautiful result of P. Castillon (\cite{Cast})
\begin{thm}\label{Castillon} Let $(M^2,g)$  be a complete Riemannian surface. Assume that  there is some $\lambda>\frac14$ such that the \sch operator $\Delta_g+\lambda K_g$ is non negative then there is a constant $c(\lambda)$ such that for any $x\in M$ and any $R>0$:
$$\mathrm{aera} \left(B(x,R)\right)\le c(\lambda) R^2.$$
Moreover such a surface is either conformally equivalent to $\bC$ or $\bC\setminus\{0\}$.
\end{thm}
\begin{rems}\begin{enumerate}[i)]
\item The non negativity condition on the  \sch operator $\Delta_g+\lambda K_g$ is equivalent to the fact that for every $\upvarphi\in \cC_0^\infty(M)$:
$$0\le \int_M \left[|d\upvarphi|^2+\lambda K_g\upvarphi^2\right]\mathrm{dA}_g.$$
\item A similar conclusion holds under the condition that  the  \sch operator $\Delta_g+\lambda K_g$ has a finite number of negative eigenvalue, or equivalently that there is a compact set $K\subset M$ such that for any $\upvarphi\in \cC_0^\infty(M\setminus M)$:
$$0\le \int_{M\setminus K} \left[|d\upvarphi|^2+\lambda K_g\upvarphi^2\right]\mathrm{dA}_g.$$
But in that case, there is a closed Riemannian  surface $(\overline{M},\bar g)$, a finite set $\{p_1,\dots, p_\ell\}\subset \overline{M}$ and a smooth function $f\colon \overline{M}\setminus\{p_1,\dots, p_\ell\}\longrightarrow\R$ such that $(M^2,g)$ is isometric to $\left(\overline{M}\setminus\{p_1,\dots, p_\ell\}, e^{2f}\bar g\right).$ 
\item This result is optimal; indeed the hyperbolic plane has exponential volume growth and the \sch operator $\Delta_g+\frac14 K_g=\Delta_g-\frac14 $ is nonnegative.
\end{enumerate}
\end{rems}
A natural question is about a higher dimensional analogue of \tref{Castillon}. However, the proof used strongly the Gauss-Bonnet formula for geodesic balls and the regularity of geodesic circles. Hence it is not clear wether it is possible to find an interesting generalization of this theorem. We will explain how the argument of Castillon can apply in the case of 3D Cartan-Hadamard manifolds \tref{CartanH} and of rotationally symmetric manifolds \tref{symm}. In particular this last result shows that it could be tricky to find examples that invalidate an extension of \tref{Castillon} result in  higher dimension.  In the recent paper \cite{Carron2016}, we have stress that a stronger spectral condition (a kind of non negativity in $L^\infty$ of the \sch operator $\Delta-\lambda\ricm$ for some $\lambda>n-2$) implies the Euclidean volume growth estimate (\ref{EVG}) . One of this result consequence thsi result (see  \tref{global}) is the following corollary that is based on a result of B. Devyver (\cite{DevKato}):
\begin{cor}If $(M^n,g)$ is a complete Riemannian manifold of dimension $n>2$ that satisfies the Euclidean Sobolev inequality \begin{equation*}
\forall \psi\in \cC_0^\infty(M)\colon\ \mu \left(\int_M \psi^{\frac{2n}{n-2}}\dv_g\right)^{1-\frac 2n}\le \int_M |d\psi|_g^2\dv_g.\end{equation*}Assume that $$\ricm\in L^{\frac{\nu_-}{2}}\cap L^{\frac{\nu_+}{2}}$$ where $\nu_-<n<\nu_+$, 
 then there is a constant $C$ such that for any $x\in M$ and $R>0$:
$$\vol B(x,R)\le C \,R^n.$$
\end{cor}

The constant C here does not only depend on the Sobolev inequality constant and the $L^{\nu_\pm/2}$ norms of $\ricm$, it depends also on the geometry on some unknown compact subset $K\subset M$.

\section{Ricci comparison} Certainly, the most famous result that leads to an Euclidean volume growth estimate (\ref{EVG}) is the Bishop-Gromov comparison theorem : {\it If $(M^n,g)$ is a complete Riemannian manifold with non negative Ricci curvature then\footnote{where $\omega_n$ is the Euclidean volume of the Euclidean unit $n-$ball.} $ \forall x\in M, \forall R>0 \ \colon\  \vol B(x,R)\le \omega_n R^n$.} From a pointwise lower bound on the Ricci curvature, one gets estimates on other geometric and analytic quantity (isoperimetric profile, heat kernel estimate, Sobolev constant, spectrum of the Laplace operator). In 1988, S. Gallot showed that some geometric estimate could also be deduced from an integral estimate on the Ricci curvature \cite{Gallot}. The volume estimate has also been proven independently by P. Li and S-T. Yau \cite{LY}. Latter on, these results has been extended by P.Petersen and G. Wei \cite{PW} and A. Aubry \cite{Aubry}. We are now going to explain how the proof of these volume estimate can be  read in order to prove \tref{logvol}.\par
\begin{proof}[Proof of theorem A ]Ê We assume that $(M^n,g)$ is a complete Riemannian manifold of dimension $n$ such that for some $\nu>n$, we have
$$\int_M \ricm^{\frac{n}{2}}d\vol_g<\infty\ \mathrm{and}\ \int_M \ricm^{\frac{\nu}{2}}d\vol_g<\infty.$$
Then for every $p\in [n,\nu]$, we have also 
$$\int_M \ricm^{\frac{p}{2}}d\vol_g<\infty.$$ Hence we can assume that $n<\nu\le n+1$.  Let $\sigma_{n-1}$ be the volume of the rounded unit $(n-1)-$sphere and for $p>n$, we define:
$$C(p,n)=2\left(\frac{p-1}{p}\right)^{\frac{p}{2}}\ \left(\frac{(n-1)(p-2)}{p-n}\right)^{\frac{p}{2}-1}.$$
Note that the integral $ \int_M \ricm^{\frac{\nu}{2}}d\vol_g$ is not scale invariant hence by scaling we can assume that:\begin{equation}\label{scaled}
C(\nu,n) \int_M \ricm^{\frac{\nu}{2}}d\vol_g=(\nu-n)^{\nu-1}\left(2^{\frac{1}{\nu-1}}-1\right)^{\nu-1}\sigma_{n-1}.\end{equation}

Indeed, we can consider $R_0^{-2}g$ in place of $g$ where $R_0$ is defined by $C(\nu,n) R_0^{\nu-n}\int_M \ricm^{\frac{\nu}{2}}d\vol_g=(\nu-n)^{\nu-1}\left(2^{\frac{1}{\nu-1}}-1\right)^{\nu-1}\sigma_{n-1}.$

Let $x\in M$ and $\exp_x\colon T_xM\rightarrow M$ be the exponential map; using polar coordinate $(r,\theta)$ in $T_xM$ (where $r>0$ and $\theta\in \bS_x=\{u\in T_xM, g_x(u,u)=1\}$, we have
$$\exp_x^*d\vol_g=J(r,\theta)drd\theta.$$ For each $\theta\in \bS_x$, there is a positive real number $i_\theta$ such that the geodesic $r\mapsto exp_x(r\theta)$ is minimizing on $[0,i_\theta]$ but not on any larger interval. If $\cU=\{(r,\theta) \in (0,+\infty)\times \bS_x,\ r<i_\theta\}$ then $\exp_x\colon \cU\rightarrow \exp_x(\cU)$ is a diffeomorphism and $$\vol_g\left(M\setminus \exp_x(\cU)\right)=0.$$
For each $\theta\in \bS_x$ the function $h(r,\theta)=\frac{J'(r,\theta)}{J(r,\theta)}$ satisfies the differential inequation of Riccati's type :
$$h'+\frac{h^2}{n-1}\le \ricm.$$
In order to compare the behavior of the volume of geodesic ball to its Euclidean counterart, P. Petersen and G. Wei introduced :
$$\Psi(r,\theta)=\left(h(r,\theta)-\frac{n-1}{r}\right)_+$$ and they showed that on $(0,i_\theta)$, we have (in the barrer sense):
$$
\Psi'+\frac{\Psi^2}{n-1}+\frac{2}{r}\Psi\le \ricm$$

From this inequality, one deduces easily that 
\begin{align*}\frac{d}{dr}\left( \Psi^{\nu-1}J\right)&\le (\nu-1) \Psi'\Psi^{\nu-2}J+h\Psi^{\nu-1}J\\
&\le  (\nu-1)\left(-\frac{\Psi^2}{n-1}-\frac{2}{r}\Psi+ \ricm \right) \Psi^{\nu-2}J+\Psi^{\nu}J+\frac{n-1}{r}\Psi^{\nu-1}J\\
&\le\left((\nu-1) \ricm\Psi^{\nu-2}-\frac{\nu-n}{n-1}\Psi^{\nu}\right)J-\frac{2\nu-1-n}{r}\Psi^{\nu-1}J\\
&\le\left((\nu-1) \ricm\Psi^{\nu-2}-\frac{\nu-n}{n-1}\Psi^{\nu}\right)J.\end{align*}
Using the inequality:
$$ab^{\nu-2}\le \frac{2}{\nu}\left(\frac{a}{\epsilon}\right)^{\nu/2}+\frac{\nu-2}{\nu}\epsilon^{\frac{\nu}{\nu-2}}b^\nu,$$ one gets:
\begin{equation}\label{DE}\frac{d}{dr}\left( \Psi^{\nu-1}J\right)\le C(\nu,n)  \ricm^{\nu/2}J.\end{equation}
We introduce now the subset of the unit sphere $\cD_r=\{\theta\in \bS_x, r<i_\theta\}$ and 
$L(r)=\int_{\cD_r} J(r,\theta)d\theta$. So that we have
$$\vol B(x,R)=\int_0^R L(r)dr.$$
From the inequality (\ref{DE}) and the fact that $\Psi(r,\theta)$ is bounded near $r=0$, one easily deduce that 
\begin{equation}\label{IDE}
\int_{\cD_r} \Psi^{\nu-1}(r,\theta)J(r,\theta)d\theta\le C(\nu,n)\int_M \ricm^{\nu/2}(y)d\vol_g(y).
\end{equation}
Using the fact that $r\mapsto \cD_r$ is non increasing, we easily obtain that (in the barrer sense):
$$\frac{d}{dr}\left(\frac{L(r)}{r^{n-1}}\right)\le \int_{\cD_r} \Psi(r,\theta)\frac{J(r,\theta)}{r^{n-1}}d\theta.$$
And with H\"older inequality, one arrives to 
\begin{align*}\frac{d}{dr}\left(\frac{L(r)}{r^{n-1}}\right)^{\frac{1}{\nu-1}}&\le \frac{1}{\nu-1}\left(\int_{\cD_r} \Psi^{\nu-1}(r,\theta)\frac{J(r,\theta)}{r^{n-1}}d\theta\right)^{\frac{1}{\nu-1}}\\
&\le \frac{1}{\nu-1} r^{-\frac{\nu-1}{n-1}}\, \left(C(\nu,n)\int_M \ricm^{\nu/2}(y)d\vol_g(y)\right)^{\frac{1}{\nu-1}}.
\end{align*}
Hence one gets:
$$\left(\frac{L(r)}{r^{n-1}}\right)^{\frac{1}{\nu-1}}\le \sigma_{n-1}^{\frac{1}{\nu-1}}+\frac{1}{\nu-n} r^{\frac{\nu-n}{n-1}}\, \left(C(\nu,n)\int_M \ricm^{\nu/2}(y)d\vol_g(y)\right)^{\frac{1}{\nu-1}}.$$
With the assumption \eqref{scaled}, one gets that for any $r\in [0,1]$ then
\begin{equation}\label{volEsmallr}
L(r)\le 2\sigma_{n-1}r^{n-1}\ \mathrm{and}\ \vol B(x,r)\le 2\omega_n r^n.\end{equation}
In order to estimate the volume of balls of radius larger than $1$, we will used the same argument and get that for any $p\in (n,\nu]$ and any $r>1$:
$$\frac{d}{dr}\left( h^{p-1}J\right)\le C(p,n)  \ricm^{p/2}J.$$
This estimate was one of the key point in Gallot's work.
We let $$I_p=\int_M \ricm^{p/2}(y)d\vol_g(y)$$ and we obtain similarly:
\begin{align*}\frac{d}{dr}\left(L(r)\right)^{\frac{1}{p-1}}&\le \frac{1}{p-1}\left(\int_{\cD_r} h^{p-1}(r,\theta)J(r,\theta)d\theta\right)^{\frac{1}{p-1}}\\
&\le \frac{1}{p-1} \left(\int_{\cD_r} h^{p-1}(1,\theta)J(1,\theta)d\theta\right)^{\frac{1}{p-1}} + \frac{1}{p-1} \, \left(C(p,n)I_p\right)^{\frac{1}{p-1}}\\
 &\le\frac{1}{p-1} \left(\int_{\cD_1} h^{p-1}(1,\theta)J(1,\theta)d\theta\right)^{\frac{1}{p-1}} + \frac{1}{p-1} \, \left(C(p,n)I_p\right)^{\frac{1}{p-1}}.
\end{align*}
We have to estimate the first term, with H\"older inequality, we easily get 
$$ \left(\int_{\cD_1} h^{p-1}(1,\theta)J(1,\theta)d\theta\right)^{\frac{1}{p-1}}\le L(1)^{\frac{1}{p-1}-\frac{1}{\nu-1}}\left(\int_{\cD_1} h^{\nu-1}(1,\theta)J(1,\theta)d\theta\right)^{\frac{1}{\nu-1}},$$
 using $h(1,\theta)\le (n-1)+\Psi(1,\theta)$, we have
$$\left(\int_{\cD_1} h^{\nu-1}(1,\theta)J(1,\theta)d\theta\right)^{\frac{1}{\nu-1}}\le (n-1)L(1)^{\frac{1}{\nu-1}}+\left(\int_{\cD_1} \Psi^{\nu-1}(1,\theta)J(1,\theta)d\theta\right)^{\frac{1}{\nu-1}}.$$
But with (\ref{IDE}) and (\ref{scaled}), we have
$$\left(\int_{\cD_1} \Psi^{\nu-1}(1,\theta)J(1,\theta)d\theta\right)^{\frac{1}{\nu-1}}\le\left( C(\nu,n)I_\nu\right)^{\frac{1}{\nu-1}}=(\nu-n)\left(2^{\frac{1}{\nu-1}}-1\right)\sigma_{n-1}^{\frac{1}{\nu-1}},$$
so that 
$$\left(\int_{\cD_1} h^{\nu-1}(1,\theta)J(1,\theta)d\theta\right)^{\frac{1}{\nu-1}}\le 2^{\frac{1}{\nu-1}}(\nu-1)\sigma_{n-1}^{\frac{1}{\nu-1}},$$
and with (\ref{volEsmallr}), one gets
$$ \left(\int_{\cD_1} h^{p-1}(1,\theta)J(1,\theta)d\theta\right)^{\frac{1}{p-1}}\le 2^{\frac{1}{p-1}}(\nu-1)\sigma_{n-1}^{\frac{1}{p-1}}.$$
And we get the following inequality for $p>n$ and $r>1$:
\begin{align*}\left(L(r)\right)^{\frac{1}{p-1}}&\le \left(L(1)\right)^{\frac{1}{p-1}}+2^{\frac{1}{p-1}}\frac{\nu-1}{p-1}\sigma_{n-1}^{\frac{1}{p-1}}(r-1)+\frac{r-1}{p-1} \, \left(C(p,n)I_p\right)^{\frac{1}{p-1}}\\
&\le 2^{\frac{1}{p-1}}\sigma_{n-1}^{\frac{1}{p-1}}+2^{\frac{1}{p-1}}\frac{\nu-1}{p-1}\sigma_{n-1}^{\frac{1}{p-1}}(r-1)+\frac{r-1}{p-1} \, \left(C(p,n)I_p\right)^{\frac{1}{p-1}}\\
&\le 2^{\frac{1}{p-1}}(\nu-1)\sigma_{n-1}^{\frac{1}{p-1}}r+\frac{r-1}{p-1} \, \left(C(p,n)I_p\right)^{\frac{1}{p-1}}  \end{align*}
Using the inequality $(a+b)^{p-1}\le 2^{p-2}\left(a^{p-1}+b^{p-1}\right)$ and assuming that $n<p\le n+1$, one gets
$$L(r)\le 2^{n}n^{n}\sigma_n r^{p-1}+2^{p-2}\left(\frac{r-1}{p-1}\right)^{p-1} \, C(p,n)I_p.$$
One comes back to the definition of the constant $C(p,n)$ and we obtain the estimate
\begin{align*}
2^{p-2}\left(\frac{1}{p-1}\right)^{p-1} \, C(p,n)&=2^{p-1}\frac{1}{(p-1)^{p-1}}\left(\frac{p-1}{p}\right)^{\frac{p}{2}}\ \left(\frac{(n-1)(p-2)}{p-n}\right)^{\frac{p}{2}-1}\\
&=\frac{2^{p-1}}{p} \left(\frac{n-1}{p}\right)^{\frac{p}{2}-1}\left(\frac{p-2}{p-1}\right)^{\frac{p}{2}-1}(p-n)^{-\frac{p}{2}+1} \\
&\le\frac{2^{n}}{n}(p-n)^{-\frac{p}{2}+1}. \end{align*}
And one gets:
\begin{equation}\label{derMaj}
L(r)\le 2^{n}n^{n}\sigma_nr^{p-1}+\frac{2^{n}}{n}(p-n)^{-\frac{p}{2}+1}I_pr^{p-1}.\end{equation}
The idea is now to choose $p=n+(\nu-n)\frac{1}{\log(er)}=n+(\nu-n)\upepsilon$, where $\upepsilon=\frac{1}{\log(er)}$.
By H\"older inequality, one has
$$I_p\le I_{n}^{\frac{\nu-p}{\nu-n}}\, I_\nu^{\frac{p-n}{\nu-n}}\le  I_{n}^{1-\upepsilon}\, I_\nu^{\upepsilon}.
$$ 

We easily get the estimates 
$$r^{p-1}=r^{n-1} \exp\left( (\nu-n)\frac{\log(r)}{\log(er)}\right)\le e\, r^{n-1},$$
$$(p-n)^{-\frac{p}{2}+1}=\left(\frac{\log(er)}{\nu-n}\right)^{\frac{n}{2}-1}\,  \left((\nu-n)\upepsilon\right)^{-\frac{\nu-n}{2}\upepsilon}.$$
Using that $(\nu-n)\upepsilon\in (0,1]$ and that if $x\in (0,1]$ then $x^{-x}\le e^{\frac{1}{e}}\le 4$, one gets
 $$(p-n)^{-\frac{p}{2}+1}\le 2 \left(\frac{\log(er)}{\nu-n}\right)^{\frac{n}{2}-1}.$$
Now the second term in the right hand side of the inequality (\ref{derMaj}) is bounded above by:
$$\frac{2^{n+1}}{n} e A_{\nu,n}^\upepsilon r^{n-1}\left(\frac{\log(er)}{\nu-n}\right)^{\frac{n}{2}-1} \sigma_n \left(\frac{I_n}{\sigma_{n-1}}\right)^{1-\upepsilon}\, ,$$
where with our scaling assumption
$$A_{\nu,n}=\frac{I_\nu}{\sigma_{n-1}\, }=\frac{\left(2^{\frac{1}{\nu-1}}-1\right)^{\nu-1} (\nu-n)^{3\frac{\nu}{2}-2}} {2\left(\frac{\nu-1}{\nu}\right)^{\frac{\nu}{2}}\ \left((n-1)(\nu-2)\right)^{\frac{\nu}{2}-1}}.$$

Using $3\le n<\nu\le n+1$, one easily verifies
$$
\frac{\left(2^{\frac{1}{\nu-1}}-1\right)^{\nu-1}}{ 2\left(\frac{\nu-1}{\nu}\right)^{\frac{\nu}{2}}}=\left( 1-\frac{1}{2^{\frac{1}{\nu-1}}}\right)^{\nu-1}\left(1+\frac{1}{\nu-1} \right)^{\frac{\nu}{2}}\le 1.$$
Hence $A_{\nu,n}\le 1$ and letting  
$J:=\max\left\{ 1, \frac{I_n}{\sigma_{n-1}}\right\}$, we eventually obtain
$$L(r)\le \sigma_{n_1} r^{n-1}\left(2^{n}n^{n}+\frac{2^{n+1}}{n} e J\left(\frac{\log(er)}{\nu-n}\right)^{\frac{n}{2}-1}\right) $$
and 
$$\vol B(x,r)\le \omega_n r^{n}\left(2^{n}n^{n}+\frac{2^{n+1}}{n} e J\left(\frac{\log(er)}{\nu-n}\right)^{\frac{n}{2}-1}\right).$$
Hence we have shown that there are positive constant $\Gamma$ that depends only of $n,\nu, \int_M \ricm^{n/2}(y)d\vol_g(y)$ such that for any $r\le 1$:
$$\vol B(x,r)\le 2\omega_n r^n$$ and for any $r\ge 1$:

$$\vol B(x,r)\le \Gamma r^n\, \left(\log\left(e r\right)\right)^{\frac{n}{2}-1}.$$
\end{proof}

\begin{proof}[Proof of of first statement in \cref{CH}] The  Theorem 2.1 in \cite{CH} states that
if  $\Omega$ is a domain of $(M,g_0)$, a compact Riemannian 
manifold of dimension $n>2$ and if $g=e^{2f}g_0$
  is a complete Riemannian metric on $\Omega$  whose Ricci tensor satisfies
$$\int_{\Omega} \|\ricci_g\|^{\frac{n}{2}}(x)d\vol_g(x)<\infty$$
and such that for some point $x_0\in\Omega$:
$$\vol_g\, B(x_0,r) = o(r^n \log^{n-1} r),$$
then there is a finite set $\{p_1,...,p_k\}\subset M$ such that
$$\Omega=M-\{p_1,...,p_k\}.$$
Hence \tref{logvol} and this theorem implies the first statement of \cref{CH}.
\end{proof}
\section{Spectral assumptions}

\subsection{A formula}  The following formula is easy to show using the equation of Jacobi fields (see for instance \cite[lemme 1.2]{carma}).
\begin{lem} Let $(M^n,g)$ be a complete Riemannian manifold and let $\Sigma\subset M$ be a smooth compact hypersurface with trivial normal bundle and $\vec\nu\colon \Sigma\longrightarrow TM$ be a choice of unit normal vectors field. Let $\mathrm{I\!I}$ be the associated second fundamental form and $h=\mathrm{Tr} \,\mathrm{I\!I}$ be the mean curvature. If $\Sigma^r$ is the parallel hypersurface defined by:
$$\Sigma^r=\left\{\exp_x(r\, \vec\nu(x)); x\in \Sigma\right\}$$ then 
\begin{equation}\label{App}
\left.\frac{d^2}{dr^2}\right|_{r=0} \vol \Sigma^r=\int_{\Sigma} \left[H^2-|\mathrm{I\!I}|^2-\ricci(\vec\nu,\vec\nu)\right]d\sigma_g.\end{equation}
\end{lem}
Using the Gauss Egregium theorem, one can give another expression for formula (\ref{App}). If $R_\Sigma$ is the scalar curvature of the induced metric on $\Sigma$ and $R_M$ the scalar curvature of $M$ and $K$ is the sectional curvature of $M$, then if $(e_1,\dots ,e_{n-1})$ is an orthonormal basis of $T_x\Sigma$ then
\begin{align*}
H^2-|\mathrm{I\!I}|^2-\ricci(\vec\nu,\vec\nu)&=R_\Sigma-\sum_{i,j} K(e_i,e_j)-\sum_{i=1}^{n-1}  K(e_i,\vec \nu)\\
&=R_\Sigma-\sum_{i=1}^{n-1}\ricci(e_i,e_i).
\end{align*}
In particular, if we let $\rho(x)$ be  the lowest eigenvalue of the Ricci tensor at $x$ then we get
\begin{equation}\label{Appu}
\left.\frac{d^2}{dr^2}\right|_{r=0} \vol \Sigma^r\le \int_{\Sigma} \left[R_\Sigma-(n-1)\rho\right]d\sigma_g.\end{equation}
\subsection{The case of 3D Cartan-Hadamard manifolds}
\begin{theo}\label{CartanH} Let $(M^3,g)$  be a Cartan-Hadamard manifold such that for some $\lambda>\frac12$, the \sch operator $\Delta_g+\lambda \rho$ is non negative then there is a constant $c(\lambda)$ such that for any $x\in M$ and any $R>0$:
$$\mathrm{aera} \left(B(x,R)\right)\le c(\lambda) R^3.$$
\end{theo}
We are grateful to S. Gallot who suggests that Castillon's proof could be adapted in the setting of 3D Cartan-Hadamard manifolds.
\proof Recall that a Cartan-Hadamard manifold is a complete simply connected Riemannian manifold with non positive sectional curvature and on such a manifold the exponential map is a global diffeomorphism. In particular the geodesic sphere are smooth hypersurfaces. In the setting of \tref{CH}, we fixe $o\in M$ and consider $A(r)=\mathrm{aera}\left(\partial B(o,r)\right)$. It is a smooth function and $A(0)=0$ and $A'(0)=0$.
We define $\upxi(r)=(R-r)^\alpha$ with $\alpha>1/2$. Integrating by parts, we easily get :
$$\int_0^R A''(r)\upxi^2(r)dr=\int_0^R A(r)\left(\upxi^2\right)''(r)dr=2\, \frac{2\alpha-1}{\alpha}\int_0^R A(r)\left(\upxi'(r)\right)^2(r)dr.$$
If we define now 
$$\upvarphi_{R}(x)=\begin{cases}
(R-d(o,x))^\alpha & \mathrm{if}\ d(o,x)\le R\\
0& \mathrm{if}\  d(o,x)\ge R.
\end{cases}$$
We get 
$$2\,\frac{2\alpha-1}{\alpha} \int_M \left|d\upvarphi_R\right|^2 \dv=\int_0^R A''(r)\upxi^2(r)dr.$$
Using the formula (\ref{Appu}) and the Gauss-Bonnet formula one gets:
\begin{align*}
2\,\frac{2\alpha-1}{\alpha}\int_M \left|d\upvarphi_R\right|^2 \dv&\le 8\pi \int_0^R \upxi^2(r)dr-2\int_M \rho \upvarphi_R^2 \dv\\
&=8\pi \frac{R^{2\alpha+1}}{2\alpha+1}-2\int_M \rho \upvarphi_R^2 \dv.\end{align*}
Hence
$$
\,\left(2-\frac{1}{\alpha}\right)\int_M \left|d\upvarphi_R\right|^2 \dv+\int_M \rho \upvarphi_R^2 \dv\le 4\pi \frac{R^{2\alpha+1}}{2\alpha+1}\dv.$$
We choose $\alpha>1/2$ such that 
$$\frac{1}{\alpha}+\frac{1}{\lambda}<2,$$ the non negativity of $\Delta_g+\lambda \rho$ implies that $$0\le \frac{1}{\lambda}\int_M |d\upvarphi_R|^2\dv+\int_M\rho\upvarphi_R^2\dv,$$ and we obtain
$$\left(2-\frac{1}{\alpha}-\frac{1}{\lambda}\right) \int_M |d\upvarphi_R|^2\dv\le 4\pi \frac{R^{2\alpha+1}}{2\alpha+1},$$
with 
$$\alpha^2 \left(\frac{R}{2}\right)^{2\alpha-2}\vol B(o,R/2)\le \int_M |d\upvarphi_R|^2\dv$$
one obtains:
$$\vol B(o,R/2)\le \frac{2^{2\alpha}\pi\, \lambda}{\alpha (2\alpha+1)\left(\lambda(2\alpha-1)-1\right)} R^3.$$
\endproof

\begin{rems}\begin{enumerate}[i)]
\item A similar conclusion holds under the condition that  the  \sch operator $\Delta_g+\lambda \rho $ has a finite number of negative eigenvalue for some $\lambda>1/2$.
\item By comparison theorem, we already know that 
$$\vol B(o,R)\ge \omega_n R^3.$$
Hence in the setting of \tref{CH}, the volume of geodesic balls in uniformly comparable to $R^3$.
\item Again this result is optimal because for the hyperbolic space, the \sch operator $\Delta_g+\frac12 \rho =\Delta_g-1$ is non negative.
\end{enumerate}\end{rems}
\subsection{The case of rotationally symmetric manifold}
\begin{theo}\label{symm} We consider $\R^n$ endowed with a rotationally symmetric metric
$$(dr)^2+J^2(r)(d\theta)^2,$$
where $J$ is smooth with $J(0)=0$ and $J'(0)=1$. If for some $\lambda\ge \frac{n-1}{4}$ the \sch operator $\Delta+\lambda \rho$ is non negative, then
$$\vol B(0,R)\le c(n,\lambda) R^n.$$
\end{theo}
\proof
We let $A(r)=\vol  \partial B(0,R)$, then $$A''(r)\le \sigma_{n-1}(n-1)(n-2) f^{n-3}(r)-(n-1)\int_{ \partial B(0,r)}\rho d\sigma.$$
Using the same function $\upvarphi_R$ one gets:
$$
\left(4-\frac{2}{\alpha} \right)\int_M |d\upvarphi_R|^2\dv+(n-1)\int_M\rho\upvarphi_R^2\dv
\le \int_0^R\sigma_{n-1}\gamma_n f^{n-3}(r)(R-r)^{2\alpha}dr$$
where $\gamma_n=(n-1)(n-2)$.
But using H\"older inequality, we also have
$$ \int_0^R f^{n-3}(r)(R-r)^{2\alpha}dr\le \left[\int_0^R (R-r)^{2\alpha-2} f^{n-1}(r)  dr\right]^{\frac{n-3}{n-1}}\left[\frac{R^{2\alpha+n-2}}{2\alpha+n-2} \right]^{\frac{2}{n-1}},$$
but
$$\int_0^R (R-r)^{2\alpha-2}\sigma_{n-1} f^{n-1}(r)  dr=\frac{1}{\alpha^2}\int_M |d\upvarphi_R|^2\dv.$$ Now one chooses $\alpha>1/2$ such that 
$0<4-\frac{2}{\alpha}-\frac{n-1}{\lambda}$, and one gets
$$\left(4-\frac{2}{\alpha}-\frac{n-1}{\lambda}\right)^{\frac{n-1}{2}}\,\int_M |d\upvarphi_R|^2\dv\le \frac{\gamma_n^{\frac{n-1}{2}}}{\alpha^{n-3}}\sigma_{n-1}\frac{R^{2\alpha+n-2}}{2\alpha+n-2}. $$
And the same argumentation yields 
$$\vol B(o,R/2)\le c(n,\lambda) R^n.$$
\endproof
\subsection{With a stronger spectral assumption}
On a non compact manifold, the behavior of the heat semigroup of a \sch operator may be very different on $L^2$ and on $L^\infty$. For instance, E-B. Davies and B. Simon have studied the case of the \sch operator $L_\lambda=\Delta-\lambda V$ on the Euclidean space $\R^n$ where the potential $V$ is defined by:$$V(x)=\begin{cases}
1/\|x\|^2 & \mathrm{if}\ \|x\| \ge 1\\
0& \mathrm{if}\   \|x\| < 1.
\end{cases}$$
When $\lambda\in (0,(n-2)^2/4)$, the operator $L_\lambda$ is non negative hence for any $t>0$:
$$\left\|e^{-tL_\lambda}\right\|_{L^2\to L^2}\le1.$$
Let $\alpha=\frac{n-2}{2}-\sqrt{ \left(\frac{n-2}{2}\right)^2-\lambda}.$
According to E-B. Davies and B. Simon \cite[Theorem 14]{DS}, we have that for any $\epsilon >0$ there are positive constants $c,C$ such that 
 $$c(1+t)^{\alpha-\epsilon}\le  \left\|e^{-tL_\lambda}\right\|_{L^\infty\to L^\infty}\le C(1+t)^{\alpha+\epsilon}.$$
 Recall that a \sch operator $L$ on a non compact Riemannian manifold is non negative if and only if there is a positive function $h$ solution of $Lh=0$ (\cite{Ag,MP}). 
\begin{defi} A
\sch operator $L$ is gaugeable with constant $\gamma\ge 1$ is 
there is a $h\colon M\longrightarrow \R$ such that 
$$Lh=0\ \,\mathrm{and}\,\  1\le h\le \gamma.$$
\end{defi}
Hence if a \sch operator $L$ is gaugeable, then it is non negative. One can also show that if  \sch operator $L$ is gaugeable with constant $\gamma$, then for any $t>0$, then $$ \left\|e^{-tL_\lambda}\right\|_{L^\infty\to L^\infty}\le\gamma.$$
One can even show that if  $(M,g)$ is stochastically complete\footnote{For instance (see \cite{GriBAMS}) when one has for some
$o\in M$:
$$\vol B(o,R)\le c e^{c R^2},$$}
and if $$\sup_{t>0} \left\|e^{-tL_\lambda}\right\|_{L^\infty\to L^\infty}=\gamma$$  then $L$ is gaugeable with constant $\gamma$.
In \cite{Carron2016}, we have shown the following result
\begin{thm}\label{global} If $(M^n,g)$ is a complete Riemannian manifold of dimension $n>2$ that satisfies the Euclidean Sobolev inequality \begin{equation*}
\forall \psi\in \cC_0^\infty(M)\colon\ \mu \left(\int_M \psi^{\frac{2n}{n-2}}\dv_g\right)^{1-\frac 2n}\le \int_M |d\psi|_g^2\dv_g\end{equation*} and such that for some $\delta>0$ the \sch operator $\Delta-(n-2)(1+\delta)\ricm$ is gaugeable with constant $\gamma$ then there is a constant $\uptheta$ depending only on $n, \delta, \gamma$ and the Sobolev constant $\mu$ such that for all $x\in M$ and $R\ge 0$:
$$ \frac{1}{\uptheta}\, R^{n}\le \vol B(x,R)\le \uptheta\, R^{n}.$$
\end{thm}

\subsection{Volume growth and heat kernel estimates}
If $(M^n,g)$ is a complete Riemannian manifold, its heat kernel $H\colon (0,+\infty)\times M\times M\rightarrow (0,+\infty)$ is the Schwartz kernel of the operator $e^{-t\Delta}$:
$$\forall f\in \cC_0^\infty(M)\colon
\left(e^{-t\Delta}f\right)(x)=\int_M H(t,x,y)f(y)\dv_g(y).$$
Estimates of heat kernels is known to imply estimate on the volume of geodesic balls. For instance, the lower Gaussian bound: 
$$\forall t>0,x,y\in M\colon\ H(t,x,y)\ge  \upgamma\,t^{-\frac n 2}e^{-\frac{d(x,y)^2}{c t}}$$ implies
the (\ref{EVG}) conditions:
$$\forall x\in M, \forall R>0\colon \vol_g(B(x,R)\le c_n c^{-\frac n2}\upgamma^{-1} R^n. $$
 
This result is classical and a proof can be found in  \cite[Proof of Theorem 4.1]{C_JLMS}, we also recommend the nice survey of A. Grigor'yan \cite{Grisurvey2} about these relationships in the context of metric spaces.

\section{Conformal geometry and real harmonic analysis}
G. David and S. Semmes have introduced a refinement of the notion of Muckenhoupt $A_\infty$-weights. 
\begin{defi} A measure $d\mu=e^{nf}dx$ on $\R^n$ is said to be a \it{strong $A_\infty$-weight} if there is a positive constant $\uptheta$ such that:
\begin{enumerate}[i)]
\item for any Euclidean ball $\bB(x,R)$: $\mu\left(\bB(x,2R)\right)\le \uptheta\mu\left(\bB(x,R)\right).$
\item  if $d_f$ is the geodesic distance associated to the conformal metric $g=e^{2f}\eucl$ then for any $x,y\in \R^n$, $$d_f(x,y)^n/\uptheta \le \mu\left(\bB_{[x,y]}\right)\le \uptheta d_f(x,y)^n,$$
where $\bB_{[x,y]}$ is the Euclidean ball with diameter the segment $[x,y]$.
 \end{enumerate}
\end{defi}
\begin{rems}
\begin{enumerate}
\item If the definition, we assume that $f$ is a smooth function, but it is possible to define strong $A_\infty$-weight under the sole condition that $e^{nf}$ is locally integrable.
\item It is possible to define what is strong $A_\infty$-weight for Ahlfors regular metric measure space $(X,d,\nu)$ \cite{SemmesAin, Costea, KK}.
\end{enumerate}
\end{rems}
It turns out that conformal metrics induced by a strong $A_\infty$-weight have very nice properties.
\begin{thm}(\cite{DS}) Let $(\R^n, g=e^{2f}\eucl)$ be a conformal metric such that  $\dv_g= e^{n f}dx$ is a strong $A_\infty$ weight then there is a positive constant $\upgamma$ such that
\begin{enumerate}[i)]
\item for any $g_f$-geodesic ball $B(x,r)$ 
$$\upgamma^{-1}r^n\le  \vol_g(B(x,r))\le \upgamma r^n,$$
\item for any smooth domain $\Omega\subset \R^n$:
$$
\left(\vol_g(\Omega)\right)^{\frac{n-1}{n}} \leq  \upgamma \vol_g\left(\partial\Omega\right).$$
\end{enumerate}
\end{thm}
There are several analytic criteria on $f$ or geometric criteria on the conformal metric $g$  implying that the associated volume measure is a strong $A_\infty$ weight.
\begin{thm}\label{HARn} Let $(\R^n g=e^{2f}\eucl)$ be a conformal metric. Then any of the following hypotheses yields that $\dv_g= e^{n f}dx$ is a strong $A_\infty$ weight:
\begin{enumerate}[a)]
\item   $f=(1+\Delta)^{-s/2}v$ with $s\in (0,n)$ and $v\in L^{\frac{n}{s}}$, (\cite{BHS,BHS2}).
\item  $\int_{\R^n} |df|^n dx<\infty$, (\cite{ACT}).\par 
Let $\upgamma_n:=\frac 12 \int_{\bS^n} \mathbf{Q}_{\mathrm{rounded}}\dv_{\mathrm{rounded}}$.
\item The conformal metric  $g$ is normal and if its $\mathbf{Q}_g$-curvature satisfies
 $\int_{\R^n} \left(\mathbf{Q}_g\right)_+\dv_g< \upgamma_n$ and $\int_{\R^n} \left|\mathbf{Q}_g\right|\dv_g<\infty$, 
(\cite{W1,Wang}).
\item  The $\mathbf{Q}_g$-curvature satisfies
$\int_{\R^n} \left|\mathbf{Q}_g\right|\dv_g<\infty$, $\int_{\R^n} \left(\mathbf{Q}_g\right)_+\dv_g<\upgamma_n$  and if the negative part of the scalar curvature satisfies
$\int_{\R^n} \left(\mathrm{Scal}_g\right)_-^{\frac{n}{2}}\dv_g< \infty$, ( \cite{WW}).
\item We have $\vol(\R^n,g)=\infty$ and the scalar curvature satisfies
$\int_{\R^n} \left|\mathrm{Scal}_g\right|^{\frac{n}{2}}\dv_g< \infty,$ (\cite{ACT}).
\end{enumerate}
\end{thm}

The main steps of the proof of \tref{harmonicanalysis} are the following :
\begin{itemize}
\item[Stage 1:] Show that if $g=e^{2f}\eucl$ is such that $\int_{\R^n} |df|^n dx<\infty$, then $e^{n f}dx$ is a strong $A_\infty$ weight.
\item[Stage 2:] Study the scalar curvature equation 
$$\Delta f-\frac{n-2}{2}\left|df\right|^2=\frac{1}{2(n-1)} \scal_ge^{2f},$$ and for large $R>0$, find a solution of the equation $$\Delta \underline{f}-\frac{n-2}{2}\left|d\underline{f}\right|^2=\frac{1}{2(n-1)} \scal_ge^{2f}\un_{\R^n\setminus\bB(R)}$$ satisfying
 $d\underline{f}\in L^n$ 
\item[Stage 3:] Show that $f-\underline{f}$ is a bounded function on $\R^n$.
\end{itemize}
\section{Examples}
\subsection{}
The proof of \tref{ExCH} relies upon the follow family of metrics:
\begin{lem}\label{Buildingblock1} Let $n\ge 3$ and $R>3$, there is a warped product metric on $\R\times \bS^{n-1}$ 
$$h_R=dt^2+J_R(t)^2d\theta^2$$ such that 
\begin{itemize}
\item $\displaystyle\sigma_-(h_R)\le C(n)$
\item $\displaystyle \int_{\R\times \bS^{n-1}} \sigma_-(h_R)^{\frac n2} \dv_{h_R}\le C(n) \left(\log R\right)^{1-\frac n2}$
\item $([R,+\infty)\times \bS^{n-1}, h_R)$ and $(-\infty,-R]\times \bS^{n-1}, h_R)$ are isometric to $\left(\R^n\setminus \bB(\rho(R)), \eucl)\right)$ where
$\rho(R)<R$ and $$\lim_{R\to+\infty} \frac{\rho(R)}{R}=1.$$ 
\end{itemize}
Moreover there is a there is a smooth radial function $\upvarphi_R\in \cC^\infty(\R^n\setminus\{0\})$ such that 
\begin{enumerate}[i)]
\item $\upvarphi_R\ge 0$ 
\item  $\upvarphi_R=0$ on $\R^n\setminus\bB(\rho(R))\subset \R^n\setminus\bB(R)$ 
\item the Riemannian manifold $\left(\R\times \bS^{n-1},h_R\right)$ is isometric to  $\left(\R^n\setminus\{0\}, e^{2\upvarphi_R}\mathrm{eucl}\right)$. 
\end{enumerate}
\end{lem}
\begin{proof}[Proof of \lref{Buildingblock1}] We start by the following observation:
If $h=dt^2+J(t)^2d\theta^2$ is a warped product on $\R\times \bS^{n-1}$ such that $J''\ge 0 $ and $|J'|\le 1$ then
$$\sigma_-(h)=\frac{J''}{J}.$$
Indeed the curvature operator of $h$ has two eigenvalues $$-\frac{J''}{J}\ \mathrm{and}\ \frac{1-J^2}{J^2}.$$

We consider a convex even function $\ell\colon\R\longrightarrow\R_+$ such that  
\begin{itemize}
\item $\ell(0)=1$.
\item If $|t|\ge 2$ then $\ell'(t)=\log|t|.$
\end{itemize}
It is easy to see that for any $t\in [-R,R]$
$$\ell(t)\le 1 +|t|\log R.$$
Hence there is a constant $\upgamma$ such that for any $R\ge 2$:
\begin{equation}\label{majinn2}
\int_{-R}^R \left(\ell''(t)\right)^{\frac n2}\left(\ell(t)\right)^{\frac n2-1}dt\le\upgamma \left(\log R\right)^{\frac n2}. 
\end{equation}
And for $R\ge 2$, we defined  $j_R\colon\R\longrightarrow\R_+$ by
$$j_R(t)=\begin{cases}
\ell(t)/\log(R)&\ \mathrm{if}\  |t|\le R\\
t-a(R)&\ \mathrm{if}\  |t|\ge R.
\end{cases}$$
where 
$a(R)=R-\frac{\ell(R)}{\log R}.$
By definition, $j_R$ is a $\cC^1$ function that is smooth on $\R\setminus\{-R,R\}$. The metric $k_R=dt^2+j_R(t)^2d\theta^2$ on $\left(\R\setminus\{-R,R\}\right)\times\bS^{n-1}$ satisfies
$$\sigma_-(k_R)(t,\theta)\le \frac{\ell''(t)}{\ell(t)} \le  \sup_{t\in \R} \ell''(t)=\sup_{t\in [0,2]} \ell''(t).$$
and with the estimation (\ref{majinn2}) we get that
$$\int_{\left(\R\setminus\{-R,R\}\right)\times\bS^{n-1}}\sigma_-(k_R)^{\frac n2} \dv_{k_R}\le  \sigma_{n-1}\gamma \left(\log R\right)^{1-\frac n2}.$$
Now $([R,+\infty)\times \bS^{n-1}, k_R)$ and $(-\infty,-R]\times \bS^{n-1}, k_R)$ are isometric to $\left(\R^n\setminus \bB(\rho(R)), \eucl)\right)$ where
$$\rho(R)=j_R(R)=\frac{\ell(R)}{\log R}$$ but we  have $$\ell(R)=\ell(2)+\int_2^R \log(t)dt=\ell(2)+R\log R-R-2\log 2+2.$$
and $\ell(2)\le 1+2\log 2$ hence for $R> 3$ we have $\ell(R)<R$ and we also have 
$$\lim_{R\to +\infty} \frac{\rho(R)}{R}=1.$$
We  now regularize the function $j_{R}$ while preserving the properties of the warped product metric. Let $\chi\in \cC^\infty(\R)$ such that
$\int_1^1 \chi(x)dx=1$ and 
$$\chi(t)=\begin{cases}1&\ \mathrm{if}\ t\le -1\\
0&\ \mathrm{if}\ t\ge 1.
\end{cases}$$
Let $\updelta\in (0,1)$, we let $S_\delta$ to be the even function defined by 
$$S_\delta(x)=\begin{cases}1&\ \mathrm{if}\ 0\le x\le 1-\updelta\\
\chi\left(\frac{x-1}{\updelta}\right)&\ \mathrm{if}\ x\in [1-\updelta, 1+\updelta]\\
0&\ \mathrm{if}\ x\ge 1+\updelta.
\end{cases}$$
Then $T_\updelta$ is the odd function defined by 
$$T_\updelta(x)=\int_0^x S_\updelta(\xi)d\xi.$$
We have 
$T_\updelta(x)=x$ if $|x|\le 1-\updelta$ and $$T_\updelta(x)=1-\updelta+\int_{1-\updelta}^{1+\updelta} \chi\left(\frac{\xi-1}{\updelta}\right)d\xi=1\ \ \mathrm{if}\ x\ge 1+\updelta.$$
We consider
$$J_{\updelta,R}(t)=\frac{1}{\log R}+\int_0^t T_\updelta\left(\frac{\ell'(\tau)}{\log R}\right)d\tau.$$
If $R^{1-\updelta}\ge 2$, then we have
\begin{itemize}
\item On  $[-R^{1-\updelta}, R^{1-\updelta}]$, $J_{\updelta,R}(t)=j_R(t)$.
\item On $[R^{1+\updelta},+\infty)$, $J_{\updelta,R}(t)=t-R^{1+\updelta}+J_{\updelta,R}\left(R^{1+\updelta}\right)$.
\end{itemize}
Moreover we always have
\begin{itemize}
\item $J_{\updelta,R}(t)\le j_{R}(t)$,
\item $\left|J'_{\updelta,R}(t)\right|\le 1$,
\item $0\le J''_{\updelta,R}(t)$,
\item If $t\ge 0:$ then $J''_{\updelta,R}(t)\le \frac{\ell''(t)}{ \log R}$.
\end{itemize}
Hence for the smooth metric $g_{\updelta, R}=(dt)^2+J_{\updelta,R}^2(t)d\theta^2$, we have
\begin{enumerate}[i)]
 \item For $t\ge 0$: $$\sigma_-(g_{\updelta, R})(t,\theta)=\frac{J''_{\updelta,R}(t)}{J_{\updelta,R}(t)}\le \frac{J''_{\updelta,R}(t)}{J_{\updelta,R}(0)} \le  \sup_{t\in \R} \ell''(t)=\sup_{t\in [0,2]} \ell''(t).$$
 \item 
\begin{align*}\int_{\R \times\bS^{n-1}}\sigma_-(g_{\updelta, R})^{\frac n2} \dv_{g_{\updelta, R}}&\le   2\sigma_{n-1}\int_0^{R^{1+\updelta}}  \left(J''_{\updelta,R}(t)\right)^{\frac n 2}\left(J_{\updelta,R}(t)\right)^{\frac n 2-1 } dt\\
&\le C_n (1+\updelta)^{\frac n2-1}\left(\log R\right)^{1-\frac n2}.\end{align*}
 \item $([R^{1+\updelta},+\infty)\times \bS^{n-1}, g_{\updelta,R})$ and $\left((-\infty,-R^{1+\updelta}]\times \bS^{n-1}, g_{\updelta,R}\right)$ are isometric to $\left(\R^n\setminus \bB(\rho(\updelta,R)), \eucl)\right)$ where
$ \rho(\updelta,R)=J_{\updelta,R}\left(R^{1+\updelta}\right)\le j_{R^{1+\updelta}}(R^{1+\updelta})<R^{1+\updelta}$ and
$$\lim_{R\to +\infty} \rho(\updelta,R)/R^{1+\updelta}=1$$
\end{enumerate}

It remains to demonstrate the last assertion. The radial function $\varphi_{\updelta,R}$ is defined by the equations
$$e^{\upvarphi_{\updelta,R}(r)}r=J_R\ \mathrm{and}\ e^{\upvarphi_{\updelta,R}(r)}\frac{dr}{dt}=1,$$
We easily get that 
$$\frac{d}{dt} \upvarphi_{\updelta,R}(r(t))=\frac{J_{\updelta,R}'-1}{J_{\updelta,R}}.$$ 
This implies that the function $r\mapsto \upvarphi_{\updelta,R}(r)$ is non increasing, hence choosing a solution of this differential equation that is zero for large positive $r$ we get that for all $r>0$: $\upvarphi_{\updelta,R}(r)\ge 0$.\par
The smooth metric $h_R$ will be defined by
$$h_R=g_{\updelta, T}$$ where 
$3^{\frac{1-\updelta}{1+\updelta}}=2$ and $T=R^{\frac{1}{1+\updelta}}$.
\end{proof}

\subsection{Proof of \tref{ExCH}}
As $(\bS^n\setminus\{N\},\mathrm{can})$ is conformally equivalent to $(\R^n,\mathrm{eucl})$. We are going to show that there is an infinite set $\Sigma\subset \R^n$ and a complete conformal metric $g=e^{2f}\mathrm{eucl}$ on $\R^n\setminus \Sigma$ with sectional curvature bounded from below and such that 
$$\int_{\R^n\setminus \Sigma} \sigma_-(g)^{\frac n2} \dv_g<\infty.$$
 We find a sequence of Euclidean balls $\left\{\bB(x_k,R_k)\right\}_k$ such that:
 \begin{itemize}
 \item $\displaystyle \forall k\colon R_k\ge 3$,
 \item $\displaystyle \sum_k \log(R_k)^{1-\frac n2}<\infty$,
 \item $\displaystyle \forall \ell\not=k\colon \bB(x_\ell,2R_\ell)\cap \bB(x_k,2R_k)=\emptyset$.
 \end{itemize}
 From the \lref{Buildingblock1}, for each $k$, one can find a smooth non-negative function $\upvarphi_k\in \cC^\infty_0(\R^n\setminus \{x_k\})$ such that 
 \begin{itemize}
 \item  $\upvarphi_k=0$ outside  $\bB(x_k,R_k)$,
  \item $(\R^n\setminus \{x_k\}, e^{2\upvarphi_k}\mathrm{eucl})$ is isometric to the metric $h_{R_k}$. \end{itemize}
 If we define $f=\sum \upvarphi_k$ then the Riemannian metric $g=e^{2f}\mathrm{eucl}$ satisfies the conclusion of \tref{ExCH}.
 
\endproof
\subsection{}
The proof of \tref{Exvol} relies on the following familly of Riemannian metric
\begin{lem}\label{Buildingblock2}Let $n\ge 3$ and $R \ge 3$ and $\underline{R}\in \left(\frac{1}{\log R}, +\infty\right)$. There is a   warped product metric on $\R^n$ 
$$g_{\underline{R},R}=dt^2+L_{\underline{R},R}(t)^2d\theta^2$$ such that 
\begin{itemize}
\item $\displaystyle\sigma_-(g_{\underline{R},R})\le C(n)$
\item $\displaystyle \int_{\R^n} \sigma_-(g_{\underline{R},R})^{\frac n2} \dv_{g_{\underline{R},R}}\le C(n) \left(\log R \right)^{1-\frac n2}$
\item $(\left[0,\frac{\pi}{2}\underline{R}\right]\times \bS^{n-1},g_{\underline{R},R} )$ is isometric to a rounded hemisphere of radius $\underline{R}$. \item There is some $r\in (R,2R+\pi\underline{R})$ such that $([r,+\infty)\times \bS^{n-1},g_{\underline{R},R} )$ is isometric to $\left(\R^n\setminus \bB(\rho(R)), \eucl\right)$.
\item The diameter of the ball $\{t\le r\}$ is bounded from above by $2\pi (\underline{R}+R).$
\end{itemize}
Moreover there  is a smooth non-negative function $\uppsi_{\underline{R},R}\in \cC^\infty_0(\R^n)$ such that 
$\uppsi_{\underline{R},R}=0$ on $\R^n\setminus\bB(R)$ and such that the Riemannian manifold $\left(\R^n,g_{\underline{R},R}\right)$ is isometric to  $\left(\R^n, e^{2\uppsi_{\underline{R},R}}\mathrm{eucl}\right)$.
\end{lem}

\begin{proof}[Proof of \lref{Buildingblock2}] Let $R\ge 3$ .
Let $R\ge 3$. Let $\tau\in (0,R\,)$, define $\underline{R}$ by 
$$\underline{R}=\frac{\ell(\tau)}{\sqrt{\log^2(R)-(\ell'(\tau))^2}}.$$
It is easy to show that $\tau\mapsto \underline{R}$ is increasing between  $1/\log R$ and $+\infty$.
So that any $\underline{R}> 1/\log R$ determines a unique $\tau\in (0,R)$.\par
Let  $\theta\in (\pi/2,\pi)$ be defined by 
$$\left\{\begin{array}{l}
\underline{R}\sin\left(\theta \right)=\frac{\ell(\tau)}{\log R}\\ \\\
-\cos\left(\theta\right)=\frac{\ell'(\tau)}{\log R}\end{array}
\right.$$

\begin{enumerate}
\item On $\left[0,\theta \underline{R}\right]$, we let $$L_{\underline{R},R}(t)=\underline{R}\sin\left(\frac{t}{\underline{R}}\right).$$
\item On $\left[\theta \underline{R},+\infty\right)$, we let 
$$L_{\underline{R},R}(t)=j_{R}\left(t- \theta \underline{R}-\tau\right).$$
 \end{enumerate}
 By construction, $L_{\underline{R},R}$ is smooth on $(0,+\infty)\setminus\{\theta \underline{R}, R+\tau+\theta\underline{R}\}$ and $\cC^1$ on $(0,+\infty)$.  
 We introduce the warped product metric on $\R^n$:
 $$g_{\underline{R},R}=(dt)^2+L^2_{\underline{R},R}(t)d\theta^2.$$
  It is easy to show that  
 that the sectional curvature  of $g_{\underline{R},R}$ is uniformly bounded from below and that
 $$\int_{ \left((0,+\infty)\setminus\{\theta \underline{R}, R+\tau+\theta\underline{R}\}\right)\times \bS^{n-1}} \sigma^{\frac n2}_-(g_R)\dv_{g_R}\le C(n)\left(\log R\right)^{1-\frac n2}.$$
If we let $r=R+\tau+\theta\underline{R}$, then by definition $([r,+\infty)\times \bS^{n-1},g_{\underline{R},R} )$ is isometric to $\left(\R^n\setminus \bB(\rho(R)), \eucl\right)$.
As before, one can smooth the metric $g_{\underline{R},R}$ while keeping the geometric properties.
\end{proof}
\subsection{Proof of \tref{Exvol}} Let $R\ge 9$. We can find in $\bB(0,4R)\setminus \bB(0,R)$ $N(R)$ disjoint balls $\bB(x_i,2\sqrt{R})$, with $$c_n \left( \log R \right)^{\frac n2 -1}\le N(R)\le C_n \left( \log R \right)^{\frac n2 -1}.$$
 When consider the function $f_R(x)=\sum_i \uppsi_{R,\sqrt{R}}(x-x_i)$. 
 By construction the conformal metric $g_R=e^{2f_R} \mathrm{eucl}$ satisfies:
  $$\int_{ \R^n} \sigma_-(g_R)\dv_{g_R}\le C(n) \left( \log R \right)^{\frac n2 -1}\times \left(\log R\right)^{1-\frac n2}\le C(n).$$
Moreover, the $g_R$-diameter of the Euclidean ball $\bB(0,4R)$ is less that 
$4R+2\pi(R+\sqrt{R})\le 20R$ hence 
 $$\vol_{g_R}\left( B(0, 20R)\right)\ge c_n \left( \log R \right)^{\frac n2 -1}\frac{ \sigma_n}{2} R^n.$$

\endproof

\subsection{Conformal metric on $\R^n$} The same idea leads easily to the following examples:
\begin{thm}Let $n\ge 3$. For any sequence $\left(a_k\right)_{k\in\N}$ such that
$$\sum_k a_k k^{1-\frac n2}<\infty$$
there is a complete conformal metric $g=e^{2f}\mathrm{eucl}$ whose sectional curvature is bounded from below and such that

$$\int_{\R^n}  \sigma_-(g)^{\frac n2} \dv_{g}<\infty$$ and such that
for all $k\in \N:$
$$\vol B(o,2^k)\ge c(n) a_k \left( 2^k\right)^n$$

\end{thm}

\end{document}